\documentclass{article}
\usepackage[pdftex]{hyperref}
\usepackage{tikz-cd}

\hypersetup{
  colorlinks   = true, 
  urlcolor     = gray, 
  linkcolor    = red, 
  citecolor   =  blue
}
\usepackage{amsmath,amsfonts,amsthm,amssymb,graphicx,xcolor, faktor,xfrac}
\usepackage[all,2cell,ps]{xy}

\usepackage{mathtools}

\theoremstyle{plain}
\newtheorem{thm}{Theorem}[section]
\newtheorem{lem}[thm]{Lemma}
\newtheorem{prop}[thm]{Proposition}
\newtheorem{cor}[thm]{Corollary}

\theoremstyle{definition}
\newtheorem{defn}[thm]{Definition}

\newtheorem{rem}[thm]{Remark}

\newcommand{\Z}{\mathbb Z}

\newcommand{\Q}{\mathbb Q}

\newcommand{\C}{\mathbb C}

\newcommand\sO{{\mathcal O}}

\DeclareMathOperator{\Pic}{Pic}

\DeclareMathOperator{\Alb}{Alb}

\DeclareMathOperator{\supp}{Supp}

\newcommand{\ssm}{\smallsetminus}

\newenvironment{pf}{\begin{proof}}{\end{proof}}

\title{Moving Seshadri Constants, and Coverings of Varieties of Maximal Albanese Dimension}
 \author{\small{Luca F. Di Cerbo} \\ \scriptsize{University of Florida} \\ \footnotesize{\textsf{ldicerbo@ufl.edu}} \and 
\small{Luigi Lombardi}\footnote{Partially supported by  SIR 2014 AnHyC: ``Analytic aspects in complex and hypercomplex geometry'' (code RBSI14DYEB), Grant 261756 of the Research Council of Norway, and the Simons Foundation.} \\ \scriptsize{University of Milan}\\ 
\footnotesize{\textsf{luigi.lombardi@unimi.it}}}
\date{}

\begin{document}

\maketitle

\begin{abstract}
Let $X$ be a smooth projective complex variety of maximal Albanese dimension, and let $L \to X$ be a big line bundle. We prove that the moving Seshadri constants of the pull-backs of $L$ to suitable finite  abelian \'etale covers of $X$  are arbitrarily large. As an application,  given any integer  $k\geq 1$,  there exists an abelian \'etale cover $p\colon X' \to X$ such that the adjoint system $\big|K_{X'} + p^*L \big|$  separates $k$-jets away from the augmented base locus of $p^*L$, and the exceptional locus of the pull-back of the Albanese map of $X$ under $p$.

\end{abstract}
\vspace{8cm}

\tableofcontents

\section{Introduction}\label{Introduction}

The moving Seshadri constants $\varepsilon_{\rm mov}(L ; x)$ of a big line bundle $L$ on a smooth projective variety are  asymptotic invariants attached to the sequence $\{ L^{\otimes m} \}_{m\geq 1}$ (\emph{cf}. for instance \cite{Nak02} and \cite{Ein09}). They extend  the definition and the geometric properties of the  Seshadri constants $\varepsilon (L ; x)$ attached to  nef line bundles, as introduced by Demailly in \cite{Dem90}, to the setting of big line bundles.  For instance, analogously to $\varepsilon (L ; x)$, the moving Seshadri constants control the rates of growth of the orders of jets that are separated by the big line bundles $L^{\otimes m}$ at a point $x$ as a function of $m$.  

In this paper, we prove that the  moving Seshadri constants $\varepsilon_{\rm mov}(L ; x)$  of a big line bundle on a smooth projective complex variety of maximal Albanese dimension are arbitrarily large on suitable  abelian \'etale coverings of the variety. In other words, they are \emph{virtually} unbounded, \emph{i.e}, unbounded up to a sequence of \'etale covers. A similar problem was asked by Hwang  (\emph{cf}. \cite[Problem 2.6.2]{Di Rocco}) in the setting of ample line bundles on smooth projective varieties with large algebraic fundamental group up to (not necessarily abelian) regular \'etale covers. Hwang's original problem was answered affirmatively in \cite[Theorems 1.3 and 1.4]{DiC19}, and furthermore generalized to the case of big and nef line bundles (\emph{cf}. \emph{loc. cit.} Theorem 1.7). Here we focus on varieties of maximal Albanese dimension, and we study this problem within the class of abelian \'etale covers induced by the Albanese map.  The main result of this paper is the following.

\begin{thm}\label{Mainmoving}
Let $X$ be a smooth projective variety such that the Albanese map $\alpha \colon X \to \Alb(X)$ is generically finite onto its image, and let $L$ be a big line bundle on $X$. For any integer $N>0$,   there exists an isogeny of abelian varieties  $a\colon A' \to \Alb(X)$ of finite degree  together with a commutative diagram   
\begin{equation}\label{maindiag}
\begin{tikzcd}
X'  \arrow[r, "{\alpha'}"] \arrow[d, "p"]
& A' \arrow[d, "a"] \\
X \arrow[r, "\alpha" ]
& \Alb(X)
\end{tikzcd}
\end{equation}
such that  $\varepsilon_{\rm mov}(p^*L ; x) \, \geq \, N $
for any $x\notin p^{-1}\mathbf{B}_+(L) \, \cup \, \mathbf{Exc}(\alpha')$.
\end{thm}

Here $\mathbf{B}_+(L)$ denotes the augmented base locus of $L$, while  $\mathbf{Exc}(\alpha')$ stands for  the exceptional locus of $\alpha'$, \emph{i.e.}, the union of all positive-dimensional fibers. An ingredient of the proof of Theorem \ref{Mainmoving} is the fact that the augmented base locus of  a big line bundle  behaves well under  pull-backs  of finite morphisms. In order to show this fact, in Lemma \ref{Lemma1} we will employ the description of $\mathbf{B}_+(L)$ in terms  of currents, together with a result of Favre \cite{Fav99} describing the local behavior of the pull-back of a current under a surjective map between smooth varieties.

We apply Theorem \ref{Mainmoving} in order to study the positivity of adjoint big line bundles on varieties of maximal Albanese dimension up to  abelian  \'etale covers. This is in the spirit of \cite{Hwang-To}, \cite{Yeung1} and \cite{Wang}, where the authors study this problem for other classes of manifolds up to non-abelian   regular \'etale covers. 
\begin{thm}\label{maincor}
Let $X$ be a smooth projective variety such that the Albanese map $\alpha \colon X \to \Alb(X)$ is generically finite onto its image, and let $L$ be a big line bundle on $X$.
\begin{enumerate}
\item[(i)] There exists an \'etale cover $p\colon X' \to X$ as in \eqref{maindiag} such that for every $P\in \Pic^0(X')$ the linear series $\big|K_{X'} +  p^*L + P \big|$ is very ample away from $p^{-1}\mathbf{B}_+(L) \cup \mathbf{Exc}(\alpha')$.
\item[(ii)] For any integer $k\geq 0$, there exists an \'etale cover $p \colon X' \to X$ as in \eqref{maindiag} such that for every $P\in \Pic^0(X')$ the linear series  $\big| K_{X'} + p^*L + P \big|$ separates $k$-jets at any point $x\notin p^{-1}\mathbf{B}_+(L) \cup \mathbf{Exc}(\alpha')$.  
\end{enumerate}
\end{thm}
Thus adjoint big line bundles on varieties of maximal Albanese dimension acquire  more and more positivity on certain regular \'etale covers of large degree. In fact, for unconditional results concerning the same type of positivity, higher multiples of  $K_X+L$ are needed. For instance, by \cite[Theorem 5.8]{PP}, 
the systems  $\big|3(K_X+L)\big|$ are  very ample away from $\mathbf{Exc}(\alpha)$ if $L$ is a  big and nef line bundle on a smooth projective variety of maximal Albanese dimension.

When applied to the canonical bundle $L=K_X$ of a variety of general type, Theorem \ref{maincor} assumes a slightly stronger formulation thanks to the main result of \cite{BBP13}. In this way, we recover and extend \cite[Theorem 4.1]{BPS1} by partly describing the locus  where the system $\big|2K_{X'} + P \big|$  is very ample, or  it separates $k$-jets.
\begin{cor}\label{maincor2}
Let $X$ be a smooth projective variety of general type such that the Albanese map $\alpha \colon X \to \Alb(X)$ is generically finite onto its image.
\begin{enumerate}
\item[(i)] There exists an \'etale cover $p\colon X' \to X$ as in \eqref{maindiag} such that for every $P\in \Pic^0(X')$ the linear series $\big|2K_{X'} + P \big|$ is very ample away from $\mathbf{Exc}(\alpha')$.
\item[(ii)] For any $k\geq 0$, there exists an \'etale cover $p \colon X' \to X$ as in \eqref{maindiag} such that for every $P\in \Pic^0(X')$ the linear series $\big|2K_{X'} + P \big|$ separates $k$-jets at any point $x\notin \mathbf{Exc}(\alpha')$.  
\end{enumerate}
\end{cor}
For unconditional  results,  by \cite[Theorem A]{JLT} we know  that only the tri-canonical map of a variety of general type and maximal Albanese dimension is birational onto the image. Moreover, in \cite[Theorem A]{BLNP}, the authors characterize the primitive varieties (from the point of view of generic vanishing theory)  of  general type and maximal Albanese dimension for which the bi-canonical map is not birational (\emph{cf}. also \cite{CCM-L} for the case of surfaces).
Other interesting results studying the positivity of linear systems on varieties of maximal Albanese dimension up to abelian \'etale covers are collected, among others,  in \cite[Theorem 4.1]{CJ}, \cite[Theorem B]{Luigi}, and \cite[Theorem 3.7]{BPS2}.

The techniques to establish  Theorem \ref{maincor} and Corollary \ref{maincor2}, and hence Theorem \ref{Mainmoving}, are somehow very  different from the techniques employed in \emph{loc. cit.}, which rely, among other things,  on generic vanishing theory, and the use of the eventual paracanonical map.   
Instead our arguments rely on the estimation of moving Seshadri constants along  towers of coverings  converging to the universal Albanese cover (\emph{cf}. \S \ref{conv}), and the relationship occurring between the largeness of moving Seshadri constants and separation of jets. Our approach elaborates a circle of ideas developed in \cite{DiC19} for varieties with large fundamental group, whose roots lie in the fundamental work of J. Koll\'ar \cite{KP}.
 
Throughout the paper, we work over the field of the complex numbers. At the same time, we wonder whether Theorem \ref{Mainmoving}, or its corollaries, extends to positive characteristic. In this regard, the paper \cite{Mur}  establishes the connection between the largeness of  Seshadri constants and  separation of jets of line bundles   over any algebraically closed field.

\subsection*{Acknowledgments} 
We thank Rob Lazarsfeld and Christian Schnell for their interest in this work and fruitful conversations. We also thank Rita Pardini for pointing out the references \cite{BPS1}, \cite{BPS2}, and for constructive comments. LFDC thanks the Mathematics Department at Stony Brook University for the ideal research environment he enjoyed at the beginning of this research project. He also gratefully acknowledges the start up fund of the University of Florida for support during the final stages of this work. LL thanks the University of Florida and the Max Planck Institute for Mathematics in Bonn, for financial supports and   excellent working conditions provided.

\section{Moving Seshadri Constants and Base Loci}\label{Preliminaries}
 
In this section, we describe  the  behavior of   augmented and restricted base loci under finite maps.   As an application, 
we extend \cite[Theorem 1.7]{DiC19} to the setting of  big line bundles and moving Seshadri constants.  Throughout this section, we denote by $X$ a smooth projective   variety, and by $D$ a   $\Q$-divisor.

\subsection{Augmented and Restricted Base Loci}
We denote by ${\rm Bs}(mD)_{\rm red}$ the base locus of the linear series $\big|mD\big|$ equipped with the reduced structure. The \emph{stable base locus} of $D$ is the Zariski-closed set
\[ \mathbf{B}(D) \; \; \stackrel{{\rm def}}{=} \;  \; \bigcap_m  \, {\rm Bs} (mD)_{{\rm red}}, \] 
where the intersection is over all positive integers $m$ such that $mD$ is   integral.    

\begin{defn}\label{loci}
	\begin{enumerate}
\item[(i)] The \emph{augmented base locus} of $D$ is the Zariski-closed set 
	\[ 	\mathbf{B}_+(D) \; \stackrel{ {\rm def} }{=} \; \bigcap_A \, \mathbf{B}(D \,  - \, A), \] 
	where the intersection is  over all ample $\Q$-divisors $A$. 
\item[(ii)] 	The \emph{restricted base locus} of $D$ is defined as
	\[ \mathbf{B}_-(D) \; \stackrel{ {\rm def} }{=} \; \bigcup_A \, \mathbf{B} (D \, + \,  A), \] 
	where the union is again over all ample $\Q$-divisors $A$.
\end{enumerate}
\end{defn}
Both $\mathbf{B}_+(D)$ and $\mathbf{B}_-(D)$ depend only on the numerical class of $D$. 
The  locus $\mathbf{B}_-(D)$ is a countable union of closed sets whose closure is contained in $\mathbf{B}_+(D)$.
The divisor $D$ is ample (\emph{resp.} big) if and only if $\mathbf{B}_+(D)=\emptyset$ (\emph{resp}. $\mathbf{B}_+(D)\neq X$). 
Moreover, $D$ is nef if and only if $\mathbf{B}_-(D) = \emptyset$. 
Finally, we note the series of inclusions $\mathbf{B}_-(D) \subseteq \mathbf{B}(D) \subseteq \mathbf{B}_+(D)$.
We refer to \cite{Nakamaye} for further properties about the augmented and restricted base loci. 
 
 We state two lemmas concerning the behavior of the loci $\mathbf{B}_+(D)$ and $\mathbf{B}_-(D)$ under pull-backs.
Let us begin with the restricted base locus (also called \emph{non-nef locus} in \cite{Bou04}).

 
\begin{lem}\label{Lemma0}
Let $p \colon  X' \to  X$ be a surjective morphism of smooth projective varieties. If $L$ is a
big line bundle on $X$, then $p^{-1} \mathbf{B}_- (L) = \mathbf{B}_- (p^*L)$.
\end{lem}
\begin{pf}
The proof follows easily once we describe the restricted base locus of  $L$ \emph{\`a la}  Boucksom \cite{Bou04}. By Proposition 3.6 in \cite{Bou04}, given a big line bundle $L$ and a closed positive $(1, 1)$-current with minimal singularities $T_{\rm min}\in [L]$, we have
\[
\mathbf{B}_{-}(L)=\{ \, x\in X \; \big| \; \nu(T_{\rm min}, x)>0 \, \},
\] 
where $\nu(T_{\rm min}, x)$ denotes the Lelong number of $T_{\rm min}$ at the point $x\in X$. 
By standard compactness properties of positive currents, one can always find a current with minimal singularities in any big (or even pseudo-effective) cohomology class. Moreover, two currents with minimal singularities have  the same Lelong numbers as they locally differ by a $(1,1)$-current of the form $\partial \overline{\partial}\varphi$ with $\varphi \in L^{\infty}$. For further details please refer to \cite[Section 2.8]{Bou04}. 

The pull-back  $p^* T_{\rm min} \in [p^* L]$ of the current $T_{\rm min}$ is again a closed positive current with minimal singularities (\emph{cf}. \cite[Proposition 1.12]{BEGZ10}).  
Hence we have
\[ \mathbf{B}_-(p^*L) \; = \; \{ \, x' \in X' \; \big| \;\nu(p^*T_{\rm min}, x') \; > \; 0 \, \}. \]
At this point we employ the following local result of Favre \cite[Theorem 2 or Corollary 4]{Fav99}. If  $f \colon (\C^{m}, 0) \rightarrow (\C^{n}, 0)$ is an holomorphic map generically of  maximal rank equal to n, then there exists a constant $C = C(f)>0$ depending only on $f$ such that 
\[ \nu(T, 0) \; \leq \; \nu(f^*T, 0) \; \leq \; C \, \nu(T, 0) \]
for any closed positive $(1, 1)$-current $T$ on $\C^n$. 
As $p$ is holomorphic and surjective (and therefore of generic maximal rank), we conclude that 
\[ \nu(p^*T_{\rm min}, x') \;  > \; 0 \;  \iff \;  \nu(T_{\rm min}, p(x')) \; > \; 0, \]
for any $x'\in X'$.  
\end{pf}

By means of Lemma \ref{Lemma0}, we can show that the augmented base locus has a similar property under \emph{finite} maps.

\begin{lem}\label{Lemma1}
Let $p \colon  X' \to  X$ be a finite surjective morphism of smooth projective varieties. If $L$ is a big line bundle on $X$, then 
 $p^{-1} \mathbf{B}_+ (L) = \mathbf{B}_+ (p^* L)$.
\end{lem}
\begin{pf} 
By \cite[Proposition 1.21]{Nakamaye}, for all sufficiently small ample $\Q$-divisors $A$ on $X$ such that $L-A$ is a $\Q$-divisor, we have that 
\begin{equation}\label{EL}
\mathbf{B}_{-}(L-A)=\mathbf{B}_{+}(L-A)=\mathbf{B}_{+}(L).
\end{equation} 
Now, let $\{A_{i}\}$ be a sequence of such ample $\Q$-divisors converging to zero. Since the map $p\colon X^\prime\rightarrow X$ is finite, by Nakai--Moishezon criterion  we know that the $\Q$-divisors in the sequence $\{p^*A_{i}\}$ are ample and clearly converging to zero. Thus, for $i$ large enough, we have
\begin{align}\label{uno}
\mathbf{B}_{-}(p^*(L-A_{i}))=\mathbf{B}_{-}(p^*L-p^*A_{i})=\mathbf{B}_{+}(p^*L-p^*A_{i})=\mathbf{B}_{+}(p^*L).
\end{align}
Since the cone of big divisors is open, for $i$ big enough we have that $L-A_i$ is big. Note that Lemma \ref{Lemma0} extends to $\Q$-divisors. Thus, for $i$ large enough, we have
\begin{align}\label{due}
\mathbf{B}_{-} \big(p^*(L-A_{i}) \big) = p^{-1} \big(\mathbf{B}_{-}(L-A_i) \big)=p^{-1}\mathbf{B}_{+}(L-A_i)=p^{-1}\mathbf{B}_{+}(L).
\end{align}
By combining Equations \eqref{uno} and \eqref{due}, we obtain the desired identity.
\end{pf}

\begin{rem}
In the case of big and nef line bundles, a different proof of Lemma \ref{Lemma1} was given in \cite[Lemma 4.1]{DiC19}.
\end{rem} 

 \subsection{Moving Seshadri Constants}
 
The \emph{moving Seshadri constant} $\varepsilon_{\rm mov}(D ; x)$ is a measure of  the local positivity of a big divisor. It agrees with the usual  Seshadri constant $\varepsilon(D ; x)$ if the divisor  $D$ is nef, and moreover it characterizes the augmented base locus  $\mathbf{B}_+(D)$ as the set of points at which $\varepsilon_{\rm mov}(D ; x)$ vanishes (\emph{cf}. \cite[Section 6]{Ein09}). In order to define $\varepsilon_{\rm mov}(D ; x)$, we first recall the definition of the usual Seshadri constant.
 
The Seshadri constant of an integral nef divisor $F$ at a point $x\in X$ is the non-negative real number
$$\varepsilon(F ; x) \;  \stackrel{{\rm def}}{=} \;  \max \{ \, \varepsilon\geq 0 \; \big| \; \mu^*F \, - \, \varepsilon E \, \mbox{ is nef} \, \},$$
where  $\mu \colon \widetilde X \to X$ is the blow-up at $x$ with exceptional divisor $E$. 
By homogeneity, the definition of  $\varepsilon(F ; x)$   extends to  $\Q$-divisors.  

 \begin{defn} 
Let $D$ be a $\Q$-divisor.  The moving Seshadri constant of $D$ at $x\in X$ is 
\begin{equation*}
   \varepsilon_{\rm mov}(D ; x)   =
    \begin{cases*}
   \sup_{ f^*D \, =\,  A \, + \,  E } \;  \varepsilon(A ; x) & if $x \notin \mathbf{B}_+ (D) $ \\
   0   & if $x \in \mathbf{B}_+ (D).$
    \end{cases*}
  \end{equation*}
 The supremum  is taken over all projective morphisms $f \colon  \widetilde X \rightarrow X$ with $\widetilde X$ smooth such that $f$ is an isomorphism near $x$, and all decompositions $f^* D =  A + E$ where $A$ and $E$ are, respectively, an ample and effective $\Q$-divisors on $\widetilde X$ such that  $f^{-1}(x)$ is not contained in the support of $E$.  
\end{defn}

Both quantities $\varepsilon(D ; x)$ and $\varepsilon_{\rm mov}(D ; x)$ only depend  on the numerical equivalence class of $D$. 
We now need to recall two definitions from group theory.  

 \begin{defn}
Let $\Gamma$ be a finitely generated group and $H$ be a subgroup. We say that $\Gamma$ is $H$\emph{-separable} (or that $H$ is separable in $\Gamma$) if for every element $g \in \Gamma \ssm H$ there is a subgroup $K\subset \Gamma$ of finite index such that $H\subset K$ and $g\notin K$. 
Moreover, we say that $\Gamma$ is \emph{residually finite} if the trivial subgroup $ {\rm id}_{\Gamma} $ is separable in $\Gamma$. 
\end{defn}
  
We can now prove the main result this section. It extends \cite[Theorem 1.7]{DiC19} to the case of big line bundles. We refer to \cite[Definition 1.7]{KP} for the definition of large fundamental group.  

\begin{thm}\label{movingext}
Let $X$ be a smooth projective variety and $L$ a big line bundle. Assume that either the fundamental group $\pi_1 (X)$ is residually finite and large, or  the algebraic fundamental group  $\pi^{\rm alg}_1(X)$ is large. Then for any integer $N>0$, there exists an \'etale cover $p \colon  X' \rightarrow X$ such that 
$\varepsilon_{\rm mov}(p^*L ; x) \geq N$ for any $x \notin p^{-1} \mathbf{B}_+ (L) = \mathbf{B}_+ (p^*L)$.
\end{thm}

\begin{pf} 
By \cite[Remark 1.5]{DiC19} it is enough to prove the theorem in the case  $\pi_1^{\rm alg}(X)$ is large.
We consider the universal algebraic cover $\widehat p \colon \widehat X \to X$  associated to the kernel $\widehat \Gamma$ of the natural homomorphism $\pi_1(X) \to \pi_1^{\rm alg}(X)$. As $\widehat \Gamma$ is separable in $\pi_1(X)$, there is  a sequence $\{\Gamma_j\}_{j\geq 0}$  of nested finite index normal subgroups of $\pi_1(X)$ such that $\Gamma_0=\pi_1(X)$ and  $\bigcap_{j\geq 0} \Gamma_j = \widehat \Gamma$. 
Consider the sequence 
\begin{align}\notag
\cdots \; \longrightarrow \; X_{j+1} \; \stackrel{p_{j+1}}{\longrightarrow} \; X_j \; \stackrel{p_{j}}{\longrightarrow} \; X_{j-1} \; \stackrel{p_{j-1}}{\longrightarrow} \; \cdots \; \stackrel{p_{1}}{\longrightarrow} \; X
\end{align} 
 of  regular  coverings associated to $\{\Gamma_j\}_{j\geq 0}$.

Let $q_{k}\colon X_k\rightarrow X$ be the regular covering given by the composition $q_k=p_1\circ\ldots\circ p_k$. Also, let  $L\sim_{\Q} A + E $ be a decomposition of the big line bundle $L$ in $\Q$-divisors where $A$ is ample and $E$ effective. Moreover, let $N$ be an arbitrary positive integer.
By \cite[Theorem 1.3]{DiC19}, there exists  a positive integer $k=k(N, A)$ 
   such that 
 $\varepsilon(q^*_j A ; x )\geq N$ for any $j\geq k$. Thus, by the definition of moving Seshadri constant, we have
\begin{align}\notag
\varepsilon_{\rm mov}(q^*_j L ; x) \; \geq \; \varepsilon(q^*_j A ; x) \; \geq \;  N
\end{align}
for any index $j \geq k$ and $x \notin q^{-1}_j (E)$. 

Next, let  $L \sim_{\Q} A_i + E_i$ be finitely many decompositions of $L$ in $\Q$-divisors with $A_i$ ample and $E_i$ effective as above, such
that $$\mathbf{B}_+ (L) \; =  \; \bigcap_{i=1}^l \supp(E_i)$$
(\emph{cf}. \cite[Definition 1.2 and Remark 1.3]{Nakamaye}).
By the argument above, for any integer $i \in [ 1, l ]$ we may find an integer $k_i$ such that 
 $$\varepsilon_{\rm mov}(q^*_j L ; x) \; \geq \; \varepsilon(q^*_j A_i ; x) \; \geq \; N \quad \mbox{ for any } \quad  j\geq k_i \quad \mbox{ and } \quad x \notin q^{-1}_j (E_i).$$ 
 Define $K:=\max \{k_1, \ldots , k_l \}$.  In view of Lemma \ref{Lemma1}, we observe that for any index $j\geq 1$ and  $x \notin \mathbf{B}_+ ( q^*_j L)$ there exists  an index $s_j \in [ 1,  l ]$ such that $x \notin q^{-1}_j (E_{s_j})$. 
Hence for any $j \geq K$ and  $x \notin \mathbf{B}_+ (q^*_j L)$, we compute
\begin{align}\notag
\varepsilon_{\rm mov}(q^*_j L ; x) \; \geq \; \varepsilon(q^*_j A_{s_j} ; x) \; \geq \; N.
\end{align}

%
%
%
\end{pf}

\section{Varieties of Maximal Albanese Dimension}\label{Main Section}

In this section we prove Theorem \ref{Mainmoving}.

\subsection{\'{E}tale Covers Induced by the Albanese Variety}

Let $X$ be a smooth projective   variety of dimension $n$. The Albanese variety of $X$ is an abelian variety of dimension $g:=h^0(X,\Omega_X^1)$ defined as:
$$A \; = \; \Alb(X) \;  = \; \faktor{H^0(X, \Omega_X^1)^*}{H_1(X,\Z)_{{\rm t.f.}}},$$  where $H_1(X,\Z)_{{\rm t.f.}}$ denotes the torsion free part of $H_1(X,\Z)$.
Up to the choice of a point, integration of holomorphic $1$-forms on $X$ defines an Albanese map
$$\alpha \colon X \longrightarrow A$$ such that  
$\alpha^* H^0(A , \Omega_A^1) = H^0(X,\Omega_X^1)$. By the definition of the Albanese variety, there is an identification $\pi_1(A) = H_1(X,\Z)_{{\rm t.f.}}$ so that the 
 natural  homomorphism
$$ \alpha_\# \colon  \pi_1(X)\longrightarrow \pi_1(A) $$ 
is surjective.
It follows that any irreducible abelian \'etale cover $\pi \colon B \to A$ of $A$ induces via the fiber product construction
\[
\begin{tikzcd}
Y \arrow[r, "\beta"] \arrow[d, " \rho "]
& B \arrow[d, "\pi"] \\
X \arrow[r, "\alpha" ]
& A,
\end{tikzcd}
\]
an \'etale cover $\rho \colon Y \to X$ of $X$ with the same properties.  We refer to  $Y$ (or $\rho$) as the pull-back of  $B$ along  $\alpha$. 
 
\begin{defn}
Let $\Gamma$ be a residually finite group. We say that a sequence $\{\Gamma_i\}_{i\geq 0}$ of nested, normal, finite index subgroups  of $\Gamma$ is a \emph{cofinal filtration}  if
$\bigcap^{\infty}_{i=0} \Gamma_i \; = \;  {\rm id}_{\Gamma}$.

\end{defn} 
 
\begin{lem}\label{Basic Construction}
The group $\ker (\alpha_{\#})$ is separable in $\pi_1(X)$. Moreover, for any tower of coverings
$$\cdots \; \longrightarrow X_{i+1} \; \stackrel{s_{i+1}}{\longrightarrow} \; X_i \; \stackrel{s_i}{\longrightarrow} \; X_{i-1} \; \stackrel{s_{i-1}}{\longrightarrow}  \; \cdots \; \stackrel{s_1}{\longrightarrow}  X,$$
obtained by pulling-back along $\alpha$ a tower of coverings 
$$\cdots \; \longrightarrow A_{i+1} \; \stackrel{r_{i+1}}{\longrightarrow} \; A_i \; \stackrel{r_i}{\longrightarrow} \; A_{i-1} \; \stackrel{r_{i-1}}{\longrightarrow} \; \cdots \; \stackrel{r_1}{\longrightarrow}  A$$
associated to a cofinal filtration of $\pi_1(A)$, we have
\[ \ker( \alpha_\# ) \; = \; \bigcap^{\infty}_{i=1} \pi_1 (X_i). \]

\end{lem}
\begin{pf}
We set $$\Gamma \, =  \, \pi_1(X) \quad  \mbox{and}\quad  \Lambda \, = \,  \pi_1(A) \, \simeq \, \Z^{2g}, $$ so that $A = \C^g / \Lambda$.  Let $\gamma\in \Gamma \ssm \ker(\alpha_\#)$ be an arbitrary element.  Hence $\alpha_{\#}(\gamma) \neq {\rm id}_{\Lambda}$, and there exists a finite index subgroup $\Lambda_1\leq \Lambda$ such that $\alpha_{\#}(\gamma)  \notin\Lambda_1$. Here we are using the fact that finitely generated free Abelian groups are  residually finite. Thus, let $a_1 \colon A_1 \to A$ 
be the \'etale cover associated to $\Lambda_1$, and let $s_1$
be the pull-back of $A_1$ along $\alpha$:
\[
\begin{tikzcd}
X_1 \arrow[r, "\alpha_1"] \arrow[d, "s_1"]
& A_1 \arrow[d, "r_1"] \\
X \arrow[r, "\alpha" ]
& A.
\end{tikzcd}
\]
By standard covering space theory, we note the following equality of groups 
\[
(s_1)_\# \big(\pi_1(X_1) \big) \; = \; \alpha^{-1}_\# \big((r_1)_\#\pi_1(A_1) \big).
\]
Hence the group $\Gamma_1 \stackrel{{\rm def}}{ = } \pi_1(X_1)$ contains the subgroup $\alpha_{\#}^{-1} ({\rm id}_{\pi_1(A)}) = \ker ( \alpha_{\#})$
and avoids $\gamma$ by our choice of $\Lambda_1$.
Moreover, the natural homomorphism  $$(\alpha_1)_{\#} \colon \pi_1(X_1) \to \pi_1(A_1)$$ is surjective as both $(s_1)_{\#}$ and $(r_1)_{\#}$ are injective homomorphism.
By reiterating this process with  $\alpha_1$ in place of $\alpha$, we generate a cofinal filtration $\{ \pi_1(A_i) \}_{i=1}^{\infty}$ of  $\pi_1(A)$ from which the lemma follows.

\end{pf}

\subsection{Convergence to the Universal Albanese Cover}\label{conv}

We continue to denote by $X$ a smooth projective variety, and by $\alpha \colon X \to A$ the Albanese map.
Moreover, we fix an ample line bundle   $L$  on $X$ and denote by $\omega\in [c_1(L)]$ the induced smooth K\"ahler metric. 
Following Lemma \ref{Basic Construction}, we set 
$$\Gamma \; = \; \Gamma_0  \, \stackrel{{\rm def}}{=} \, \pi_1(X) \quad \mbox{ and } \quad \Gamma_i \stackrel{{\rm def}}{=} \pi_1(X_i) \subset \Gamma.$$ 
Furthermore, we denote by $\overline{q}_i \colon X_i \to X$ the composition $s_1 \circ \ldots \circ s_i$ so that there is a commutative diagram
\begin{equation}\label{qi}
\begin{tikzcd}
X_i \arrow[r, "\alpha_i"] \arrow[d, "\overline{q}_i"]
& A_i \arrow[d, "a_i"] \\
X \arrow[r, "\alpha" ]
& A
\end{tikzcd}
\end{equation}
where $a_i = r_1 \circ \ldots \circ r_i$.  
Finally, we denote by $$\overline{q} \colon \overline{X}\longrightarrow X$$ the pull-back of the universal cover of $A$ along $\alpha$:
\begin{equation}\label{Albanese}
\begin{tikzcd}
\overline{X} \arrow[r, "\overline{\alpha}"] \arrow[d, "\overline{q}"]
& \C^g \arrow[d, "\overline{a}"] \\
X \arrow[r, "\alpha" ]
& A.
\end{tikzcd}
\end{equation}
Throughout the paper, we refer to the cover $\overline{q}\colon\overline{X}\longrightarrow X$ defined in \eqref{Albanese} as the \emph{universal Albanese cover} of $X$. Notice that, up to a finite cover, the universal Albanese cover coincides with the \emph{universal Abelian cover} of $X$. Indeed, these two infinite covers are the same if and only if $H_1(X,\Z)_{{\rm t.f.}}=H_1(X,\Z)$.

Equivalently, the cover $\overline{q}$ can   be defined as the  regular covering associated to the separable normal subgroup  $\ker(\alpha_\#)\lhd \pi_1(X)$. Hence, by Lemma \ref{Basic Construction}, we have 
$$\overline{\Gamma} \, \stackrel{{\rm def}}{=} \, \pi_1(\overline{X}) \, = \, \ker(\alpha_\#) \; = \;  \bigcap_{i\geq 1} \Gamma_i.$$ 
In Theorem \ref{Wallach}, we will show that the sequence of covers $\overline q_i\colon X_i \to X$ 
\emph{converges} to the universal Albanese cover in a precise way.  First, we note that
the coverings $X_i$ can be described as quotients of $\overline{X}$: 
\begin{equation}\label{pi2}
\begin{tikzcd}
\overline{X} \arrow[r, "\overline{\alpha}"] \arrow[d, "\overline{p}_i"]
& \C^g \arrow[d, "\overline{a}_i"] \\
X_i \arrow[r, "\alpha_i" ]
& A_i
\end{tikzcd}
\end{equation}
where 
$$X_i \; \simeq \; \faktor{\overline{X}}{ \, \overline{\Gamma}_i} \quad \mbox{ and }\quad \overline{\Gamma}_i \; \stackrel{{\rm def}}{=} \;  \faktor{\Gamma_i}{ \, \overline{\Gamma}}.$$ 
Secondly, we equip each $X_i$ with the K\"ahler metric $\overline{q}^*_i \omega$, and   $\overline X$ with $\overline q^*\omega$. Finally, we define the following quantities 
\begin{align}\notag
\overline r_i \; \stackrel{{\rm def}}{=} \; \inf\big\{ \, d(z, \overline{\gamma}_i z)\;  \big| \; z \in \overline{X},\quad  \overline{\gamma}_i \in \overline{\Gamma}_i, \quad \overline{\gamma}_i  \, \neq \,  {\rm id}_{\overline \Gamma_i}  \, \big\},
\end{align}
where the distance $d(-,-)$ is measured with respect to the metric $\overline q^*\omega$ on $\overline{X}$.

\begin{thm}\label{Wallach}
Let  $$\overline{p}_i \colon \big(\overline{X}, \overline{q}^{\, *}\omega \big) \longrightarrow \big( X_i, \overline{q}^{\, *}_i \omega \big)$$ be the Riemannian covering maps induced by the inclusions
$\overline \Gamma_i \subset  \overline{\Gamma}_0$ as in \eqref{pi2}. Then, for any $z \in\overline{X}$,  the maps
\begin{align}\label{iso}
\overline{p}_i \colon B\Big(z; \frac{\overline{r}_i}{2} \Big)  \; \longrightarrow \; \overline{p}_i \Big(B \Big(z; \frac{\overline{r}_i}{2} \Big) \Big)
\end{align}
are isometries  and 
\begin{align}\label{infinite}
\lim_{i \to \infty}\overline{r}_i \, = \, \infty.
\end{align}
\end{thm}

\begin{pf}
 By definition of the numerical invariant $\overline{r}_i$, the map $\overline{p}_i \colon  B\big(z; \frac{\overline{r}_i}{2} \big) \rightarrow \overline{p}_i \big( B\big( z; \frac{\overline{r}_i}{2} \big) \big)$ is a biholomorphism for any $z \in \overline{X}$, and since we are pulling back the metric $\omega$ on $X$, it is an isometry. 

Now we prove \eqref{infinite}. We proceed by contradiction and assume that there  exist a positive constant $M$,  and  infinite sequences $\{ z_i \} \subset \overline{X}$ and $\{\overline{\gamma}_i \} \subset  \overline{\Gamma}_i$, such that $d(z_i, \overline{\gamma}_i z_i) \leq 2M$ and $\overline{\gamma}_i \neq {\rm id}_{{\overline \Gamma}_i}$. Let $D$ be a fundamental domain for $X$ in $\overline{X}$. Thus, $D\subset \overline{X}$ is a connected open set such that $\overline{q}\colon D \to X$ is injective and $\overline{q}\colon  \overline{D} \rightarrow X$ is surjective, where $\overline{D}$ is the closure of $D$ in $\overline{X}$. Thus, for any $i$, there exists an element $g_i \in  \overline{\Gamma}_0$ such that $g_i z_i \in\overline{D}$. Let us define $z'_i=g_i z_i$ and $\overline{\gamma}'_i = g_i \overline{\gamma}_i g^{-1}_i$. Since $\overline{\Gamma}_i$ is a normal subgroup of $\overline{\Gamma}_0$, we have that $\overline{\gamma}'_i \in \overline{\Gamma}_i$. By compactness of $\overline{D}$, there exists a subsequence $\{ z'_{i_j} \}$ converging to a point $\overline{z}\in\overline{D}$. Now since 
\[
d(z'_i, \overline{\gamma}'_i z'_i) \; = \;  d(g_i z_i , g_i \overline{\gamma}_i z_i ) \; = \;  d(z_i , \overline{\gamma}_i z_i),
\]
we have that 
\begin{align}\notag
d( \overline z, \overline{\gamma}'_{i_j} \overline{z} ) \;  \leq \; d(\overline{z}, z'_{i_j}) \, + \, 2M. 
\end{align}
Since $d(\overline{z}, z'_{i_j} ) \, \rightarrow \, 0$, we then conclude that, up to a subsequence, 
$\overline{\gamma}'_{i_j} \overline{z}$ converges to a point $w \in B(\overline{z} ; 2M + \varepsilon)$ for some $\varepsilon>0$. This implies that 
\begin{align}\notag
\overline{q} (\overline{z} ) \; = \; \overline{q}(\overline{\gamma}'_{i_j}  \overline{z}) \, \longrightarrow \, \overline{q}(w).
\end{align}
Thus, there exists $\overline{\gamma} \in \overline{\Gamma}_0$ such that $\overline{\gamma} w = \overline{z}$. We therefore conclude 
\begin{align}\notag
(\overline{\gamma}'_{i_j} \cdot \overline{\gamma}) w = \overline{\gamma}'_{i_j} \overline{z} \, \longrightarrow \, w.
\end{align}
Now the action of $\overline{\Gamma}_0$ on $\overline{X}$ is properly discontinuous, so that $\overline{\gamma}'_{i_j} \cdot \overline{\gamma}= {\rm id}_{\overline{\Gamma}_{i_j}}$ for all $j$ sufficiently large. Thus, we must have $\overline{\gamma}= {\rm id}_{\overline{\Gamma}_0}$ since 
$\bigcap \overline \Gamma_{i_j} = {\rm id}_{\overline{\Gamma}_0}$, which then implies the contradiction $\overline{\gamma}'_{i_j} = 
{\rm id}_{  \overline{\Gamma}_{i_j} }$.  

\end{pf}

\subsection{Virtual Unboundedness of Moving Seshadri Constants}
We keep notation as in  the previous subsection. In addition, we assume that  $X$  is  of  \emph{maximal Albanese dimension}, namely that the Albanese map 
$$\alpha \colon X \longrightarrow A$$ is generically finite onto the image.  
We denote by $\mathbf{Exc}(-)$ the exceptional locus of a morphism that is generically finite onto its image, \emph{i.e.}, the union of all its positive-dimensional fibers.

\begin{lem}\label{Basic Construction2}
For any integer $N>0$, there exists a positive  integer $k(N)$ such that 
$$ \Big( \, ( \overline{q}^*_i L )^{\dim Z} \cdot Z  \, \Big)  \; \geq \; N $$
for any $i\geq k(N)$  and irreducible subvariety $Z \subset X_i$ not entirely contained in $\mathbf{Exc}(\alpha_i)$.
\end{lem}

\begin{pf}
We claim that for any integer $i$, and irreducible subvariety $Z\subset X_i$ not entirely contained in $\mathbf{Exc}(\alpha_i)$, we must have $Z\subsetneq \overline{p}_i \big(B(z; \frac{\overline{r}_i}{2} ) \big)$ for any ball $B( z; \frac{\overline{r}_i}{2} )$ in $\overline{X}$. We proceed by contradiction and  assume that this is not the case.  By Theorem \ref{Wallach}, we can then find a copy of $Z$ inside $\overline{X}$. By considering the commutative diagram
\[
\begin{tikzcd}
\overline{X} \arrow[r, "\overline{\alpha}"] \arrow[d, "\overline{p}_i"]
& \C^g \arrow[d, "\overline{a}_i"] \\
X_i \arrow[r, "\alpha_i" ]
& A_i
\end{tikzcd}
\]
we have that $Z$ inside $\overline{X}$ is not entirely contained $\mathbf{Exc}(\overline{\alpha})$. Thus, $\overline{\alpha}(Z)$ is a compact analytic subvariety of positive dimension inside $\C^g$ which is impossible. 

Next, we  observe that all   metrics $\overline{q}^*_i \omega$ have uniformly bounded geometry. In fact, they are pull backs of a fixed smooth K\"ahler metric on the compact manifold $X$ via \'etale covers.  In particular, there exist  positive constants $r=r(L), C_1,C_2$  
such that 
\begin{align}\label{inequality}
C_1\omega_E \; \leq \; \overline{q}^* \omega \; \leq \; C_2 \omega_E
\end{align}
on any ball $B( z ; r ) \subset\overline{X}$ (here $\omega_E$ denotes the standard Euclidean K\"ahler metric on $\C^g$). Moreover, we can arrange (\ref{inequality}) to hold true for any $\overline{q}^*_i \omega$, as well as for $\overline{q}^*\omega$ on balls of the same size. Thus, for any $i$, given any irreducible subvariety $Z\subset X_i$ of pure dimension $l$, there exists a positive constant $K_l = K( r , C_1 , C_2 ; l )$ such that for any point $p \in Z$ we have   
\begin{equation}\label{volume}
{\rm Vol}_l \big(B ( p ; r ) \cap Z \big) \; \geq \; K_l.
\end{equation}
 Here the volume  is computed by the integral of the $l$-th power of $\overline{q}^*_i \omega$ over the smooth part of $Z$. The inequality \eqref{volume} follows from the inequalities \eqref{inequality} and the same statement for the euclidean metric on $\C^g$, see for example \cite[Remark 3.5]{DiC19}. Recall now that for any $i$, given a point $p \in Z$,  the subvariety $Z$ is not entirely contained in  $B(p; \frac{\overline{r}_i}{2} )$. Thus we have that $${\rm Vol}_l (Z) \;  \geq  \; \alpha_l \overline{r}_i,$$ for some constant $\alpha_l = \alpha(K_l) > 0$. Next, let us observe that by construction we have
\begin{align}\notag
(\overline{q}_i^*L)^l \cdot Z \; = \;  {\rm Vol}_ l(Z) \; \geq \; \alpha_l  \, \overline{r}_i
\end{align}
for any $i$. By Theorem \ref{Wallach}, we have that $\overline{r}_i \rightarrow \infty$ and the lemma follows. 
\end{pf}

We apply the previous result in order to show that the Seshadri constants of pull-backs of ample line bundles  under  \'etale covers coming from the Albanese variety are unbounded. In other words, they are \emph{virtually} unbounded in the class of finite abelian covers induced by the Albanese map.

\begin{lem}\label{Basic Construction3}
   For any  integer $N>0$, there exists a positive integer $k(N)$ such that 
$ \varepsilon(\overline{q}^*_i L ; x)  \; \geq \; N $ 
for any $i\geq k(N)$  and  $x \in X_i \ssm \mathbf{Exc}(\alpha_i)$.
\end{lem}

\begin{pf}

Let $n$ be the dimension of $X$ and let  $m=m(L)$ be a positive  integer such that $mL-K_X$ is ample and $mL-2K_X$ is very ample. We define $L' = mL-K_X$.
By Anghern--Siu Theorem \cite{Ang} the line bundle $K_{X}+m_0 L'$ is base point free for all $m_0 \geq \binom{n+1}{2}+1$. We therefore have that
\[
K_{X}+\Big(\binom{n+1}{2}+1\Big)L^\prime+mL-2K_X=\Big(\binom{n+1}{2}+2\Big)L^\prime
\]	
is base point free. Concluding, $m_0L^\prime$ is base point free for any $m_0\geq\binom{n+1}{2}+2$. Now define the constant
\[
N' \; \stackrel{{\rm def}}{=} \; \max_{c \,= \, 0,  \ldots , n} \left\{ \left( 1 + \binom{n+1}{2} (n-c) + \frac{n!}{c!} \right)^c \left( n+s \right)^n+1 \right\}
\]
where $s>0$ is a fixed integer  to be determined later. By Proposition \ref{Basic Construction2}, there exists  an integer $i_0 = i_0(s)$ such that 
\[ 
\Big( \, (\overline{q}_i^* L')^{\dim V }\cdot V  \, \Big) \; \geq \; N'
\]
for any $i\geq i_0$  and any subvariety $V\subseteq X_i$ not entirely contained in $\mathbf{Exc}(\overline{\alpha}_i)$. 
Moreover, by construction, we have that the line bundle
\begin{align}\notag
\overline{q}^*_i(L' - K_X) \; = \; \overline{q}^*_i L' - K_{X_i} 
\end{align}
is nef and 
\begin{align}\notag
\overline{q}^*_i (m_0 L')
\end{align}
is base point free. By a theorem of Ein--Lazarsfeld--Nakamaye \cite[Theorem 4.4]{ELN}, we know that for any $x\in X_i \ssm \mathbf{Exc}(\overline{\alpha}_i)$ the linear system $\big|K_{X_i} + \overline{q}^*_i L'\big|$ separates  $s$-jets at $x$.
Since 
\begin{align}\notag
K_{X_i } + \overline{q}^*_i L' \; = \; m \overline{q}^*_i L,
\end{align}
by \cite[Proposition 6.3]{Dem90} we conclude that $\varepsilon( \overline{q}^*_i L ; x) \geq \frac{s}{m}$. Thus, given any $N>0$, it suffices to take an integer $s=s(N)\geq 0$ such that $\frac{s}{m} \geq N$. We point out  that  the positive integer $m$ depends only on the line bundle $L$. Given this choice for $s$, for any $i \geq i_0(s)$ the associated cover $\overline{q}_i \colon  X_i \rightarrow X$ satisfies the conclusion of the theorem.  
\end{pf}

We can now prove the main theorem stated in the Introduction.

\begin{proof}[Proof of Theorem \ref{Mainmoving}]

Let  $$L \; \sim_{\Q} \; A \, + \,  E $$ be a decomposition of the big line bundle $L$ in  $\Q$-divisors with $A$  ample and $E$ effective. 
 By Lemma \ref{Basic Construction3}, given any positive constant $N$, there exists a positive integer $k=k(N, A)$ such that 
 $\varepsilon(\overline{q}^*_j A ; x )\geq N$ for any $x \notin \mathbf{Exc}(\alpha_j)$ and $j\geq k$. 
 Thus, by the definition of moving Seshadri constant, we have
\begin{align}\notag
\varepsilon_{\rm mov}(\overline q^*_j L ; x) \; \geq \; \varepsilon(\overline q^*_j A ; x) \; \geq \;  N
\end{align}
for any index $j \geq k$ and $x \notin \overline{q}^*_j (E) \cup \mathbf{Exc}(\alpha_j)$. 

Next, let  $L \sim_{\Q} A_i + E_i$ be finitely many decompositions of $L$ in $\Q$-divisors with $A_i$ ample and $E_i$ effective,   as above, 
satisfying $$\mathbf{B}_+ (L) \; =  \; \bigcap_{i=1}^l \supp(E_i).$$
By the above argument, for any integer $i \in [ 1, l ]$ we may find an integer $k_i$ such that 
 $$\varepsilon_{\rm mov}(\overline q^*_j L ; x) \; \geq \;\varepsilon(\overline{q}^*_j A_{i} ; x) \; \geq\; N \; \mbox{for any} \; x \notin \overline{q}^*_j (E_i) \cup\mathbf{Exc}(\alpha_j) \; \mbox{and} \; j \, \geq \, k_i.$$ 
 Define $K:=\max \{k_1, \ldots , k_l \}$.  In view of Lemma \ref{Lemma1}, we observe that   for any index $j\geq 1$ and  $x \notin \mathbf{B}_+ ( \overline{q}^*_j L)$ there exists  an index $s_j \in [ 1,  l ]$ such that $x \notin \overline q^{-1}_j (E_{i_j})$. 
Finally, fix any index $j \geq K$. Hence, for  any  $x \notin \mathbf{B}_+ (\overline{q}^*_j L) \cup \mathbf{Exc}(\alpha_j)$ we have
\begin{align}\notag
\varepsilon_{\rm mov}(\overline{q}^*_j L ; x) \; \geq \; \varepsilon(\overline{q}^*_j A_{ s_j} ; x) \; \geq \; N.
\end{align}
 \end{proof}
 
\subsection{Positivity of Linear Systems}

A linear system $\big|L\big|$ on a smooth projective variety $X$ is said to  separate $k$-jets at a point $x \in X$ if the natural homomorphism
$$H^0(X, L ) \; \longrightarrow \; H^0 \Big(X, L \otimes \faktor{\sO_X}{\frak{m}_x^{k+1}} \Big)$$ is surjective. 
In order to prove Theorem \ref{maincor} and Corollary \ref{maincor2} of the Introduction, we need to establish  the connection between moving Seshadri constants and separation of jets of adjoint line bundles. 
\begin{prop}\label{propseshadri}
Let $X$ be a smooth projective variety of dimension $n$ and $L$ a big line bundle.
\begin{itemize}
\item[(i)] If $\varepsilon_{\rm mov}(L ; x) > n+k$,  then $\big|K_{X}+L\big|$ separates $k$-jets at $x$.
\item[(ii)] Let $V\subsetneq X$ be a Zariski-closed subset. 
If $\varepsilon_{\rm mov}(L ; x) > 2n$ for all $x\notin V$, then $\big|K_{X}+L\big|$ is very ample away from $V$.
\end{itemize}
\end{prop}

We refer to \cite[Proposition 6.8]{Ein09} for the proof of $(i)$. 
Point $(ii)$ follows from the proof of point $(i)$ by employing the main idea of \cite[Proposition 5.1.19 (ii)]{Laz1}.

\begin{proof}[Proof of Theorem \ref{maincor}]
Set $n=\dim X$ and let $N$ be an arbitrary positive integer.
Note that if $D$ is a divisor, then $\mathbf{B}_+(D) = \mathbf{B}_+(D + P)$ and 
$\varepsilon_{\rm mov}(D  ; x) \;  = \; \varepsilon_{\rm mov}(D + P ; x)$ 
for any divisor $P$ with $c_1(P)=0$.
Hence, by Theorem \ref{Mainmoving}, there exists a commutative diagram as in \eqref{maindiag} such that 
$$\varepsilon_{\rm mov}(p^*L + P ; x) \;  = \; \varepsilon_{\rm mov}(p^*L; x) \; > \; N $$ for any $x\notin p^{-1}\mathbf{B}_+(L) \cup \mathbf{Exc}(\alpha')$ and $P\in \Pic^0(X')$. 
The corollary follows  by Proposition \ref{propseshadri}.
\end{proof}

\begin{proof}[Proof of Corollary \ref{maincor2}]
 We set $L=K_X$ in Theorem \ref{maincor}, and 
 observe that, by \cite{BBP13},  the augmented base locus $\mathbf{B}_+ (K_{X'})$ of a canonical divisor of a variety of general type is  uniruled. This yields the inclusion 
\[
\mathbf{B}_+(K_{X'}) \subseteq \mathbf{Exc}(\alpha')
\]
since any morphism from a smooth variety to an abelian variety contracts all rational curves. The proof is complete.
\end{proof}


Finally, we refine Corollary \ref{maincor2} in the case of varieties with \emph{finite} Albanese map. In this setting, the result assumes the strongest possible formulation.

\begin{cor}
Let $X$ be a smooth projective variety such that the Albanese map $\alpha \colon X \to A$ is finite onto the image.
\begin{enumerate}
\item[(i)] If $X$ is of general type, then there exists a  commutative diagram as in \eqref{maindiag} such that  for every $P\in \Pic^0(X')$ the linear system $\big|2K_{X'} + P \big|$ defines an embedding of $X'$ into a projective space.
\item[(ii)] There exists a  commutative diagram as in \eqref{maindiag} such that  the linear system $\big|2K_{X'}  \big|$ induces the Iitaka fibration. 
\end{enumerate} 
 \end{cor}
 
\begin{proof} 
The first point is application of Theorem \ref{maincor} as  $\mathbf{Exc}(\alpha') = \emptyset$. For the second point, we set $n=\dim X$.
By \cite[Theorem 13]{Kawamata}, we can find an \'etale cover $X_1\times B \rightarrow X$ such that $X_1$ is a
variety of general type of dimension $k={\rm kod}(X)$, and $B$ is an abelian variety of dimension $n-k$.  By    point $(i)$, there exists an \'etale cover $X'_1\rightarrow X_1$ such  that the map associated to the linear system $\big|2K_{X'_1}  \big|$ is an embedding. Now, the \'etale cover $X':=X'_1\times B \rightarrow X$ is such that  the linear system $\big|2K_{X'}  \big|$ induces the Iitaka fibration. 
\end{proof}


\end{document}